\documentclass[a4paper,11pt]{article}
\usepackage[toc,page]{appendix}
\usepackage{amsthm}
\usepackage{amsmath}
\usepackage{latexsym}
\usepackage{amssymb}
\usepackage{verbatim}
\usepackage{setspace}
\usepackage{qtree}
\usepackage{enumitem}
\usepackage{xcolor}
\usepackage{url}
\usepackage{stmaryrd}
\usepackage{authblk}
\usepackage{hyperref}
\usepackage[margin=1.4in]{geometry}

\title{\Huge Convergence in the space of \\ compact labeled metric spaces}

\makeatletter
\newcommand\email[2][]%
   {\newaffiltrue\let\AB@blk@and\AB@pand
      \if\relax#1\relax\def\AB@note{\AB@thenote}\else\def\AB@note{\relax}%
        \setcounter{Maxaffil}{0}\fi
      \begingroup
        \let\protect\@unexpandable@protect
        \def\thanks{\protect\thanks}\def\footnote{\protect\footnote}%
        \@temptokena=\expandafter{\AB@authors}%
        {\def\\{\protect\\\protect\Affilfont}\xdef\AB@temp{#2}}%
         \xdef\AB@authors{\the\@temptokena\AB@las\AB@au@str
         \protect\\[\affilsep]\protect\Affilfont\AB@temp}%
         \gdef\AB@las{}\gdef\AB@au@str{}%
        {\def\\{, \ignorespaces}\xdef\AB@temp{#2}}%
        \@temptokena=\expandafter{\AB@affillist}%
        \xdef\AB@affillist{\the\@temptokena \AB@affilsep
          \AB@affilnote{}\protect\Affilfont\AB@temp}%
      \endgroup
       \let\AB@affilsep\AB@affilsepx
}
\makeatother

\author[1]{Reijo Jaakkola}
\author[2]{Antti Kykkänen}
\affil[1]{Faculty of Information Technology and Communication Sciences \break Tampere University, Finland}

\email{\texttt{reijo.jaakkola@tuni.fi} \break}

\medskip

\affil[2]{Department of Mathematics and Statistics \break University of Jyväskylä, Finland}

\email{\texttt{antti.k.kykkanen@jyu.fi}}

\date{\today}

\DeclareMathOperator{\dis}{dis}

\DeclareMathOperator{\diam}{diam}

\newcommand*{\N}{\mathbb{N}}

\newcommand*{\R}{\mathbb{R}}
\newcommand{\X}{\mathcal X}
\newcommand{\Y}{\mathcal{Y}}

\newcommand{\red}[1]{\textcolor{red}{#1}}

\newcommand{\abs}[1]{\left\vert#1\right\vert}

\newcommand{\raja}[1]{\overset{\scriptscriptstyle #1}{\to}}

\newcommand*{\norm}[1]{\left\lVert#1\right\rVert}

\begin{document}

\setlength\abovedisplayskip{3pt}
\setlength\belowdisplayskip{3pt}

\theoremstyle{plain}
\newtheorem{theorem}{Theorem}[section]
\newtheorem{lemma}[theorem]{Lemma}
\newtheorem{corollary}[theorem]{Corollary}
\newtheorem{proposition}[theorem]{Proposition}
\newtheorem{fact}[theorem]{Fact}

\theoremstyle{definition}
\newtheorem{definition}[theorem]{Definition}
\newtheorem{remark}[theorem]{Remark}
\newtheorem{example}[theorem]{Example}

%%%%%%%%%%%%%%%%%%%%%%%%%%%%%%%%%%%%%%%%%%%%%%%%%%%%%%%%%%%%%%%%%%%%%%%%%%%%%%%%%%%%%%%%%
%%%%%%%%%%%%%%%%%%%%%%%%%%%%%%%%%%%%%%%%%%%%%%%%%%%%%%%%%%%%%%%%%%%%%%%%%%%%%%%%%%%%%%%%%

\maketitle

\begin{abstract}
\noindent A labeled metric space is intuitively speaking a metric space together with a special set of points to be understood as the geometric boundary of the space.
We study basic properties of a recently introduced labeled Gromov-Hausdorff distance --- an extension of the classical Gromov-Hausdorff distance --- which measures how close two labeled metric spaces are.

We provide a toolbox of results characterizing convergence in the labeled Gromov-Hausdorff distance. We obtain a completeness result for the space of labeled metric spaces and precompactness characterizations for subsets of the space. The results are applied to travel time inverse problems in a labeled metric space context.
\end{abstract}

\section{Introduction}

\emph{All the metric spaces considered in this work are assumed to be compact.} Given a set $L$ an $L$-\textbf{labeled metric space} is just a pair $(X,\alpha)$, where $X$ is a metric space and $\alpha:L \to X$ is a \textbf{labeling}, i.e., an arbitrary map. The intended interpretation for the set $\alpha(L) \subseteq X$ is that it is the \emph{geometric} boundary of $X$, but we emphasize that there are, of course, also other natural interpretations.

Motivated by geometric inverse problems, labeled metric spaces were introduced in \cite{multisourcelgh}, the main idea behind them being as follows. Suppose that we have a compact Riemannian manifold $M$ with boundary and a finite metric space $X$ which we use to approximate $M$. How should one measure how well $X$ approximates $M$? One natural option would be to use the classical Gromov-Hausdorff distance (for a definition, see \cite[Chapter 7]{metricgeometry}), which would measure how good of a metric approximation of $M$ the space $X$ is.

However, in addition to knowing how close the metric structure of $X$ is to that of $M$, in geometric inverse problems we are also interested in knowing how close the ``boundary'' of $X$ is to the boundary $\partial M$ of $M$. This of course requires us to define what the boundary of $X$ is and this is where we use labeled metric spaces. Indeed, given that $X$ is intended to approximate $M$, it is very natural to define the boundary of $X$ simply as the image of some labeling $\alpha:\partial M \to X$. Letting $\iota:\partial M \to M$ denote the inclusion, we are now interested in knowing how close the two $\partial M$-labeled spaces $(X,\alpha)$ and $(M,\iota)$ are.

In \cite{multisourcelgh} the authors introduced an extension of the GH-distance which measures in a natural way how close two labeled metric spaces are. The definition which we present here differs slightly from the one given in \cite{multisourcelgh}. However, it is straightforward to show that the two definitions are bi-Lipschitz equivalent, see Appendix \ref{appendix:equivalence-of-definitions}.

\begin{definition}\label{def:lgh-distance}
    Let $(X,\alpha)$ and $(Y,\beta)$ be two $L$-labeled metric spaces and $t\geq 0$. A pseudometric
    \[d: (X \sqcup Y)^2 \to [0,\infty)\]
    on $X \sqcup Y$ (the disjoint union of $X$ and $Y$) is called $(t,L)$-\textbf{admissible}, if it satisfies the following requirements.
    \begin{enumerate}
        \item $d\upharpoonright X \times X = d_X$ and $d\upharpoonright Y \times Y = d_Y$.
        \item $d_H(X,Y) \leq t$.
        \item \label{item:definition-of-admissible} $\sup_{\ell \in L}d(\alpha(\ell),\beta(\ell)) \leq t$.
    \end{enumerate}
    We define the \textbf{labeled Gromov-Hausdorff distance} of $(X,\alpha)$ and $(Y,\beta)$ as the following quantity:
    \[d_{GH}^L(X,\alpha;Y,\beta) := \inf\{t \mid \exists \ (t,L)\text{-admissible pseudometric on } X \sqcup Y\}.\]
\end{definition}

%\begin{remark}
%    In what follows we will speak of $d:(X\sqcup Y)^2 \to 
% [0,\infty)$ as a metric, but understand that it is only a %pseudometric.
%\end{remark}

\begin{remark}
    In Item~(\ref{item:definition-of-admissible}) of Definition \ref{def:lgh-distance} the supremum is understood as the supremum in $[0,\infty)$ whence $\sup \varnothing = 0$. Therefore, $d^\varnothing_{GH}$ is merely the classical GH-distance between metric spaces.
\end{remark}

We say that two $L$-labeled metric spaces $(X,\alpha)$ and $(Y,\beta)$ are $L$-\textbf{isometric}, if there is an isometry $h \colon X \to Y$ such that $h(\alpha(\ell)) = \beta(\ell)$ for all $\ell \in L$. In other words, there needs to be an isometry between $X$ and $Y$ which maps the ``boundary" of $X$ to the ``boundary" of $Y$. Such an isometry $h$ is referred to as an $L$-\textbf{isometry}. The main property of the LGH-distance is that $d_{GH}^L(X,\alpha;Y,\beta) = 0$ if and only if the two spaces are $L$-isometric (for a proof, see Proposition \ref{prop:l-isometric} below).

\paragraph{Our contributions.} The main purpose of this work is to develop basic theory of the LGH-distance, the main guiding principle being the heuristic that it should be possible to extend results on the GH-distance to the context of LGH-distance in a canonical way. While the LGH-distance was introduced in \cite{multisourcelgh}, its basic properties were not studied extensively in that work and we aim to fill this gap (at least partially). Our emphasis is on convergence results.

In Section \ref{sec:completeness} we establish that the space of all $L$-isometric classes, which is denoted by $\mathcal{M}_L$, is a Polish space when equipped the metric induced by the LGH-distance. Since, as mentioned above, the LGH-distance of two labeled metric spaces is zero if and only if they are $L$-isometric, LGH-distance indeed induces a metric on $\mathcal{M}_L$ (as opposed to just a pseudometric). We note that $\mathcal{M}_L$ is indeed a set; in fact, in Appendix \ref{appendix:no-isometry-classes} we establish that its cardinality is at most $2^{\max\{\aleph_0,|L|\}}$.

In Section \ref{sec:characterisations} we establish characterisations for LGH-convergence, which extend similar characterisations for GH-convergence (all of them can be found in \cite[Chapter 7]{metricgeometry}). For instance, Theorem \ref{thm:finite-net-condition-convergence} establishes that convergence of labeled spaces can be captured by convergence on finite labeled subspaces. As a necessary by-product of these characterisations, we extend several useful notions related to GH-distance, such as correspondence and approximate isometry, to the context of LGH-distance. 

In Section \ref{sec:compactness-criteria} we investigate precompactness criteria for subsets of $\mathcal{M}_L$. Our starting point is the classical result of Gromov which states that a uniformily totally bounded collection of (compact) metric spaces is precompact w.r.t. GH-distance \cite{Gromov1981}. While we are able to show that the same criteria works in the case where $L$ is countable, extending this criteria to the case where $L$ is an arbitrary set is challenging, since it turns out that any such criteria will necessarily yield some version of the classical Arzel-Ascoli theorem. We by-pass this problem by assuming that $L$ is a compact metric space\footnote{Given that $L$ is in some applications the boundary of a compact Riemannian manifold, this assumption is well-justified.} and derive two extensions of Gromov's compactness theorem.

Section \ref{sec:inverse-problems} is concerned with applications of the precompactness results from Section~\ref{sec:compactness-criteria} to geometric inverse problems. We outline a well-known general procedure for stabilizing the inversion of an problem and we introduce a metric space variant of the so-called travel time inverse problem. Travel time problems arise from geophysical questions, where one ask if the inner structure of the Earth can be recovered by measuring arrival times of earthquakes on the surface. In the mathematical formulation, we try to recover a metric space from the knowledge length of certain paths in the metric space --- data which corresponds to travel times of seismic waves. We prove a unique determination result for a travel time problem, and apply our precompactness result together with the stabilization procedure to prove stability in the determination. 

\paragraph{Future work.} In~\cite{multisourcelgh} the labeled Gromov-Hausdorff distance was invented to measure geometric convergence of a sequence of finite metric spaces with ``boundaries''. Their result is of the form: If $M$ is an unknown but sufficiently nice Riemannian manifold and we have sufficient boundary data, we can reconstruct finite metric space approximations converging to the manifold. With this application in mind, a natural question to ask is what are some general conditions for existence of a finite sequence of spaces approximating a given labeled metric space? Our $(\varepsilon,\delta)$-approximation and precompactness results can be seen as partial answers towards questions of this sort, but they leave something to be desired, being unable to prove existence sequences converging to a given space without assuming some form of convergence or smallness in distances in the first place.

The question of existence of finite approximating labeled spaces gives rise to one alternative interpretation for the labels. Suppose we are given a compact metric space $X$ and a surjective function $\alpha \colon L \to X$ for some set $L$. The restrictions of $\alpha$ to different finite subsets $L' \subseteq L$ give finite approximations $\alpha(L')$ of the space $X$. Let $(L_n)_{n \in \N}$ form an increasing sequence of finite subsets of $L$, providing finer and finer finite approximations $\alpha(L_n)$ for $X$. There are multiple natural question to be investigated. How many and how large approximations $\alpha(L_n)$ are needed to construct a countable subset $\tilde X$ of $X$ with small Hausdorff distance between $\tilde X$ and $X$? How do properties of $\alpha$ affect the constructibility of $\tilde X$? What and how much information does $\tilde X$ carry about $X$? As a compact space $X$ is separable, so in principle we should be able to construct $\tilde X$ within arbitrarily small Hausdorff distance from $X$, at least if we are given sufficiently many approximations.

%A labeling $\alpha \colon L \to X$ has also interpretations of very different nature. Consider two isometric metric spaces $X$ and $Y$. Can we separate $X$ and $Y$ by labeling one, multiple or infinite number of points? For example, let $X = Y = [0,1]$ with the standard metric and equipped $X$ and $Y$ with labelings $\alpha \colon \{0,1\} \to X$ and $\beta \colon \{0,1\} \to Y$ defined by $\alpha(0) = 0$, $\alpha(1) = 1$, $\beta(0) = 1/2$ and $\beta(1) = 1$. Then we see that $d^L_{GH}(X,\alpha;Y,\beta) > 0$. The reason for the separation is that we have chosen to label two points with different distances, which works for any isometric spaces.

%Can similar separation be accomplished by labeling points with pairwise equal distances? Are there isometric metric spaces $X$ and $Y$ with labelings $\alpha \colon L \to X$ and $\beta \colon L \to Y$, where $L$ is of any cardinality, so that $d_X(\alpha(\ell),\alpha(\ell')) = d_Y(\beta(\ell),\beta(\ell'))$, but $d^L_{GH}(X,\alpha;Y,\beta) > 0$? Ultimately, the question is an labeled metric space formulation of the question: If two metric spaces $X$ and $Y$ have the property that for any proper subsets $A \subsetneq X$ and $B \subsetneq Y$ there is an isometry $A \to B$, are the spaces $X$ and $Y$ necessarily isometric?

It would also be natural to study labeled metric spaces in a setting where we impose explicit constraints on the topology of the labelings image. As an example of the type of result this line of inquiry could produce, in Appendix \ref{appendix:compactness-property} we show that if $(X,\alpha)$ and $(Y,\beta)$ are $L$-labeled metric spaces such that $d_{GH}^{L'}(X,\alpha;Y,\beta) = 0$, for every \emph{finite} $L' \subseteq L$, then $d_{GH}^L(X,\alpha;Y,\beta) = 0$ follows from the assumption that $\alpha(L)$ is closed.\footnote{We believe that this assumption cannot be dropped, but we have failed to produce a counter-example which would demonstrate this.}

\section{Completeness of the space of labeled metric spaces}\label{sec:completeness}

In this section we introduce the notion of LGH-convergence and to establish that $\mathcal{M}_L$ is a Polish space, i.e., separable and complete. We start with the following simple (but useful) proposition, which demonstrates how LGH-distance extends the GH-distance.

\begin{comment}
\begin{proposition}
    $d_{GH}^L$ is indeed a metric.
\end{proposition}
\begin{proof}
    The only non-trivial case to check is the triangle-inequality. Let $(X,\alpha), (Y,\beta)$ and $(Z,\gamma)$ be $L$-metric spaces. Let
    \[d_1 : X \sqcup Y \to [0,\infty),\]
    and
    \[d_3 : Y \sqcup Z \to [0,\infty)\]
    be metrics. We now define a metric $d_2$ on $X \sqcup Z$ by setting that
    \[d_2(x,z) = \inf_{y \in Y}\{d_1(x,y) + d_3(y,z)\},\]
    for every $x \in X$ and $z \in Z$. Now we make the following observations.
    \begin{enumerate}
        \item Clearly $d_H(X,Z) \leq d_H(X,Y) + d_H(Y,Z)$, with respect to metrics $d_2,d_1$ and $d_3$ respectively.
        \item For every $\ell \in L$ we have that
        \[d_2(\alpha(\ell),\gamma(\ell)) \leq d_1(\alpha(\ell),\beta(\ell)) + d_3(\beta(\ell),\gamma(\ell))\]
        and hence
        \[\sup_{\ell \in L} d_2(\alpha(\ell),\gamma(\ell)) \leq \sup_{\ell \in L} d_1(\alpha(\ell),\beta(\ell)) + \sup_{\ell \in L} d_2(\beta(\ell),\gamma(\ell)).\]
    \end{enumerate}
    Thus, if $d_1$ is $(t_1,L)$-admissible and $d_3$ is $(t_3,L)$-admissible, then $d_2$ is $(t_2,L)$-admissible for some $t_2$ such that $t_2 \leq t_1 + t_3$. Taking infinum over the metrics $d_1$ and $d_3$ gives us the desired inequality.
\end{proof}
\end{comment}

\begin{proposition}
\label{prop:one-space}
Let $h \colon X \to Y$ be an isometry of metric spaces. Label $X$ and $Y$ with maps $\alpha \colon L \to X$ and $\beta \colon L \to Y$ be two labelings, where $L$ is any set. Then
\[
d^L_{GH}(X,\alpha;Y,\beta) = \sup_{\ell \in L}d_Y(h(\alpha(\ell)),\beta(\ell)).
\]
\end{proposition}

\begin{proof}
Define $d \colon (X \sqcup Y)^2 \to [0,\infty)$ by conditions
\[
d(x,y) = d_Y(h(x),y),
\quad
d|_{X \times X} = d_X
\quad\text{and}\quad
d|_{Y \times Y} = d_Y.
\]
Then $d$ is a $(t,L)$-admissible pseudometric on $X \sqcup Y$ since
\[
t = \sup_{\ell \in L}d(\alpha(\ell),\beta(\ell)).
\]
Furthermore, if $d'$ is any metric on $X \sqcup Y$, then since $X \sqcup Y$ is isometric to $Y \sqcup Y$ via $h \times id$, we have
\[
d'(\alpha(\ell),\beta(\ell))
=
d_Y(h(\alpha(\ell)),\beta(\ell))
\]
for all $\ell \in L$. This proves that
\[
d^L_{GH}(X,\alpha;Y,\beta)
=
\sup_{\ell \in L}d(\alpha(\ell),\beta(\ell)).
\]
\end{proof}

The following easy application of Proposition \ref{prop:one-space} establishes that $d_{GH}^L$ is indeed a metric on $\mathcal{M}_L$ and not just a pseudometric.

\begin{proposition}\label{prop:l-isometric}
    Two $L$-labeled metric spaces $(X,\alpha)$ and $(Y,\beta)$ are $L$-isometric if and only if $d^L_{GH}(X,\alpha;Y,\beta) = 0$.
\end{proposition}
\begin{proof}
    Suppose that $d^L_{GH}(X,\alpha;Y,\beta) = 0$. Then we also have that $d_{GH}(X,Y) = 0$. It is well-known that this entails that there is an isometry $h:X \to Y$ (see~\cite[Theorem 7.3.30]{metricgeometry}). Applying Proposition~\ref{prop:one-space} we deduce that 
    \[
    \sup_{\ell \in L}
    d_Y(h(\alpha(\ell)),\beta(\ell))
    =
    d_{GH}^L(X,\alpha;Y,\beta)
    =
    0,\]
    i.e., $h(\alpha(\ell)) = \beta(\ell)$, for every $\ell \in L$.
    
    Conversely, if $h \colon X \to Y$ is an $L$-isometry, then a pseudometric $d$ on $X \sqcup Y$ defined by conditions
    \[
    d(x,y) = d_Y(h(x),y), \quad d|_{X \times X} = d_X \quad\text{and}\quad d|_{Y \times Y} = d_Y
    \]
    is $0$-admissible proving that $d^L_{GH}(X,\alpha;Y,\beta) = 0$.
\end{proof}

\begin{remark}
    Proposition \ref{prop:l-isometric} was also established in~\cite[Proposition~5]{multisourcelgh} using the alternative definition of LGH-distance by essentially repeating the proof of the well-known fact that two metric spaces are isometric if and only if their GH-distance is zero.
\end{remark}

Let $((X_n,\alpha_n))_{n \in \N}$ be a sequence of $L$-labeled metric spaces. We say that this sequence (LGH-)converges to an $L$-labeled metric space $(X,\alpha)$ if 
\[\lim_{n \to \infty} d_{GH}^L(X_n,\alpha_n;X,\alpha) = 0.\]
The space $(X,\alpha)$ is called the \textbf{labeled Gromov-Hausdorff limit} (\textbf{LGH-limit}) of $((X_n,\alpha_n))_{n \in \N}$.

\begin{remark}\label{remark:gh-vs-lgh}
LGH- and GH-convergence are related in the following sense.
\begin{itemize}
    \item If $(X_n,\alpha_n) \to (X,\alpha)$ in the LGH sense, then $X_n \to X$ in the GH sense.
    \item If $X_n \to X$ in the GH sense and $\alpha \colon L \to X$ is a label, then for each $n$ we can find a labeling $\alpha_n \colon L \to X$ so that $(X_n,\alpha_n) \to (X,\alpha)$ in the LGH sense. Indeed, if $d_{GH}(X_n;X) < \varepsilon$, then for some $\varepsilon_0 < \varepsilon$ there exists a $\varepsilon_0$-admissible metric $d$ on $X_n \sqcup X$. In particular, for every $\ell \in L$ we can find $x_\ell \in X_n$ such that $d(x_\ell,\alpha(\ell)) \leq \varepsilon_0$, and thus we can define $\alpha_n$ by simply setting that $\alpha_n(\ell) := x_\ell$, for every $\ell \in L$.
\end{itemize}
In particular, all the geometric properties that are preserved under GH-limits are also preserved under LGH-limits.
\end{remark}

\begin{proposition}
    The space $(\mathcal{M}_L,d_{GH}^L)$ is Polish.
\end{proposition}
\begin{proof}
    The proof is essentially an extension of the proof of Proposition 43 given in~\cite[p. 296]{riemannianpetersen}.
    
    Since every compact metric space is a GH-limit of a sequence of finite metric spaces with $\mathbb{Q}$-valued metrics, it follows from Remark \ref{remark:gh-vs-lgh} that also every $L$-labeled metric space is an LGH-limit of a sequence of finite $L$-labeled metric spaces with $\mathbb{Q}$-valued metrics. Hence $\mathcal{M}_L$ is separable.
    
    To show completeness, let $((X_n,\alpha_n))_{n \in \N}$ be a Cauchy sequence. It is enough to show that some subsequence of $((X_n,\alpha_n))_{n \in \N}$ converges. We select a subsequence $((X_n,\alpha_n))_{n \in \N}$ (indexed with $\N$ for simplicity) such that $d_{GH}^L(X_n,\alpha_n;X_{n+1},\alpha_{n+1}) < 2^{-n}$, for all $n$, which exists since the original sequence is Cauchy. For each $n$ we fix a $(2^{-n},L)$-admissible metric $d_{n,n+1}$ on $X_n \sqcup X_{n+1}$. For every $n$ and $m$ we define a metric $d_{n,n+m}$ on $X_n \sqcup X_{n+m}$ by
    \[d_{n,n+m}(x_n,x_{n+m}) := \min \bigg\{\sum_{k=0}^{m-1} d_{n+k,n+k+1}(x_{n+k},x_{n+k+1}) \mid \forall k : x_{n+k} \in X_{n+k} \bigg\}.\]
    It is routine to show that $d_{n,n+m}$ is $(2^{-n+1},L)$-admissible. Using the metrics $d_{n,n+j}$ we can naturally define a metric on the set $\bigsqcup_{n} X_n$.
    \begin{comment}
    \begin{enumerate}
        \item The Hausdorff distance between $X_i$ and $X_{i+j}$ is less than $2^{-i+1}$, since
        \begin{align*}
        d_{H}(X_i,X_{i+j})
        &\le
        \sum_{k = 0}^{j-1}d_H(X_{i+k},X_{i+k+1})
        \\
        &\le
        \sum_{k = 0}^{j-1}2^{-i-k}
        \\
        &=
        2^{-i+1}(1-2^{-j})
        \\
        &<
        2^{-i+1}.
        \end{align*}

        \item Let $\ell \in L$. Now
        \begin{align*}
            \sum_{k=0}^{j-1} d(\alpha_{i+k}(\ell),\alpha_{i+k+1}(\ell)) 
            &\leq \sum_{k=0}^{j-1} 2^{-(i+k)}\\
            &= \frac{1}{2^{i+j-1}} \sum_{k = 0}^{j-1} 2^k\\
            &= \frac{2^j - 1}{2^{i+j-1}}\\
            &< 2^{-i+1}
        \end{align*}
        Hence $\sup_{\ell \in L} d_{i,i+j}(\alpha_i(\ell),\alpha_{i+1}(\ell)) < 2^{-i+1}$.
    \end{enumerate}
    \end{comment}
    
    Consider the set $\hat X$ of Cauchy sequences $(x_n)_{n \in \N}$ in the disjoint union, where $x_n \in X_n$.
    %\[\hat{X} = \{(x_n)_{n \in \N} \mid x_n \in X_n \text{ and } d(x_n,x_m) \to 0 \text{ as } n,m \to \infty\}.\]
    Since for each $n,m$ the metric $d_{n,n+m}$ is $(2^{-n+1},L)$-admissible, we have that
    \[d(\alpha_n(\ell),\alpha_{n+m}(\ell)) \leq 2^{-n+1},\]
    for every $\ell \in L$, which entails that $(\alpha_n(\ell))_{n \in \N} \in \hat{X}$. The space $\hat X$ comes with a natural pseudometric
    \[
    d((x_n)_{n \in \N},(y_n)_{n \in \N}) = \lim_{n \to \infty}d(x_n,y_n).
    \]
    Let $X$ denote the quotient obtained from $\hat{X}$ via the equivalence relation
    \[(x_n)_{n \in \N} \sim (y_n)_{n \in \N} \quad\Leftrightarrow\quad d((x_n)_{n \in \N},(y_n)_{n \in \N}) = 0.\]
    Equip $X$ with the metric induced by the pseudometric $d$. Furthermore, we can define an $L$-labeling on $X$ by setting that $\alpha(\ell) = (\alpha_n(\ell))_{n \in \N}$, where $(\alpha_n(\ell))_{n \in \N}$ denotes the equivalence class of the corresponding tuple.
    This gives us a labeled metric space $(X,\alpha)$. To see that $(X,\alpha)$ is the LGH-limit of $((X_n,\alpha_n))_{n \in \N}$ we extend the metric on $\bigsqcup_n X_n$ to a metric on $X \sqcup \bigsqcup_n X_n$ by
    \[
    d((x_n)_{n \in \N},(y_n)_{n \in \N}) = \lim_{n \to \infty} d(x_n,y_n).
    \]
    Then it is routine to check that $d_H(X_n,X) \le 2^{-n+1}$. In addition, since
    \[
    d(\alpha_n(\ell),\alpha(\ell)) = \lim_{m \to \infty}d(\alpha_n(\ell),\alpha_{n+m}(\ell)) \le 2^{-n+1}
    \]
    we deduce that $\sup_{\ell \in L} d(\alpha_n(\ell),\alpha(\ell)) \le 2^{-n+1}$ and consequently $(X_n,\alpha_n) \to (X,\alpha)$ in the LGH sense.
\end{proof}

From the previous proof we see that if $(X_n,\alpha_n) \raja{LGH} (X,\alpha)$, then there is a metric $d$ on the disjoint union
\[
X \sqcup \bigsqcup_n X_n
\]
such that that
\[
\lim_{n \to \infty}d_H(X_n,X) = 0
\quad\text{and}\quad
\lim_{n \to \infty}
\sup_{\ell \in L} d(\alpha_n(\ell),\alpha(\ell)) = 0.
\]

\section{Characterisations of LGH-convergence}\label{sec:characterisations}

In this section we develop different characterisations for LGH-convergence. As a by-product, we generalize several notions introduced for GH-distance to LGH-distance. All of our characterisations for LGH-convergence have counterparts in the theory of GH-convergence and they can be found, for example, in~\cite[Chapter~7]{metricgeometry}, the presentation of which we will follow quite closely.

\begin{definition}
    Let $(X,\alpha)$ and $(Y,\beta)$ be $L$-labeled metric spaces. A \textbf{correspondence} between $(X,\alpha)$ and $(Y,\beta)$ is a relation $R \subseteq X \times Y$ which has the following properties.
    \begin{enumerate}
        \item For every $x \in X$ there exists a $y \in Y$ such that $(x,y) \in R$.
        \item For every $y \in Y$ there exits $x \in X$ such that $(x,y) \in R$.
        \item For every $\ell \in L$ we have $(\alpha(\ell),\beta(\ell)) \in R$.
    \end{enumerate}
\end{definition}

Note that there always exists at least one correspondence between to $L$-labeled metric spaces. A natural example of a correspondence is the graph of a surjection $f:X \to Y$ which for every $\ell \in L$ maps~$\alpha(\ell)$ to~$\beta(\ell)$.

Consider two $L$-labeled metric spaces $(X,\alpha)$ and $(Y,\beta)$. Given a non-empty relation $R \subseteq X \times Y$ its \textbf{distortion} is defined by
\[\dis(R) := \sup \{|d_X(x,x') - d_Y(y,y')| : (x,y),(x',y') \in R\}.\]
The following theorem relates distortions of correspondences to LGH-distance and provides an alternative definition for the LGH-distance.

\begin{theorem}\label{theorem:distance-correspondence}
    The labeled Gromov-Hausdorff distance between two $L$-labeled metric spaces $(X,\alpha)$ and $(Y,\beta)$ is
    \[d_{GH}^L(X,\alpha;Y,\beta) = \frac{1}{2} \inf_{R}\dis(R).\]
    Here the infimum runs over all correspondences $R$ between $(X,\alpha)$ and $(Y,\beta)$.
\end{theorem}
\begin{proof}
    Suppose first that $d_{GH}^L(X,\alpha;Y,\beta) < t$. Thus for some $s < t$ there is a $(s,L)$-admissible pseudometric $d$ on $X \sqcup Y$. Consider the following relation
    \[R := \{(x,y) \mid d(x,y) < t\}.\]
    Using the fact that $d_H(X,Y) < t$ and $\sup_{\ell \in L} d(\alpha(\ell),\beta(\ell)) < t$ we see that $R$ is indeed a correspondence between $(X,\alpha)$ and $(Y,\beta)$. Triangle inequality can be used to establish that $\dis(R) < 2t$.
    
    Suppose then that $R$ is a correspondence between $(X,\alpha)$ and $(Y,\beta)$ such that $\dis(R) < 2t$. Choose some $s$ such that $\frac{1}{2} \dis(R) \leq s < t$. Define a pseudometric $d$ on $X \sqcup Y$ by setting that for every $x \in X$ and $y \in Y$ we have that
    \[d(x,y) := \inf\{d_X(x,x') + s + d_Y(y,y') \mid (x',y') \in R\}.\]
    It is straightforward (albeit tedious) to verify that $d$ is indeed a pseudometric.
    
    Note that $d(x,y) = s$, when $(x,y) \in R$. We claim that $d$ is $(s,L)$-admissible.
    \begin{enumerate}
        \item Let $x \in X$ and let $y \in Y$ be such that $(x,y) \in R$. Now $d(x,y) = s$ and hence $X \subseteq \bigcup_{y \in Y} B(y,s)$. Similarly, $Y \subseteq \bigcup_{x \in X} B(x,s)$. Thus $d_H(X,Y) \leq s$.
        \item Let $\ell \in L$. Since $(\alpha(\ell),\beta(\ell)) \in R$, we have that $d(\alpha(\ell),\beta(\ell)) = s$.
    \end{enumerate}
    Thus $d$ is indeed $(s,L)$-admissible and hence $d_{GH}^L(X,\alpha;Y,\beta) \leq s < t$.
\end{proof}

Given a mapping $f:X \to Y$ between metric spaces, its \textbf{distortion} $\dis f$ is defined to be the distortion of its graph. Intuitively speaking the distortion of $f$ measures how far away $f$ is from being an isometry.

\begin{definition}
    Let $\varepsilon > 0$. A mapping $f:(X,\alpha) \to (Y,\beta)$ is called an $(\varepsilon,L)$-\textbf{isometry}, if it satisfies the following requirements.
    \begin{enumerate}
        \item $\dis f < \varepsilon$.
        \item $\displaystyle \sup_{\ell \in L} d_Y(f(\alpha(\ell)),\beta(\ell)) < \varepsilon$.
        \item $f(X)$ is an $\varepsilon$-net in $Y$.
    \end{enumerate}
\end{definition}

\noindent The following theorem implies that if $(X_n,\alpha_n) \raja{LGH} (X,\alpha)$, then there are approximate $L$-isometries between $(X_n,\alpha_n)$ and $(X,\alpha)$.

\begin{theorem}
\label{thm:epsilom-isometries}
    Let $(X,\alpha)$ and $(Y,\beta)$ be $L$-labeled metric spaces and $\varepsilon > 0$.
    \begin{enumerate}
        \item If $d_{GH}^L(X,\alpha;Y,\beta) < \varepsilon$, then there exists a $(2\varepsilon,L)$-isometry from $X$ to $Y$.
        \item If there exists an $(\varepsilon,L)$-isometry from $X$ to $Y$, then $d_{GH}^L(X,\alpha;Y,\beta) < 2\varepsilon$.
    \end{enumerate}
\end{theorem}
\begin{proof}
    For the first claim, Theorem~\ref{theorem:distance-correspondence} guarantees that there exists a correspondence $R$ between $X$ and $Y$ with $\dis(R) < 2\varepsilon$. For every $x \in X$ we select some $y \in Y$ such that $(x,y) \in R$ and then set $f(x) := y$. Clearly, $\dis f \leq \dis(R) < 2\varepsilon$. It is also simple to verify that $f(X)$ is an $2\varepsilon$-net in $Y$. Finally, since for every $\ell \in L$, we have that $(\alpha(\ell),\beta(\ell)) \in R$, we also have that $d(\beta(\ell),f(\alpha(\ell))) = |d(\alpha(\ell),\alpha
    (\ell)) - d(\beta(\ell),f(\alpha(\ell)))| < \dis R < 2\varepsilon$.
    
    For the second claim, let $f$ be an $\varepsilon$-isometry. Define a relation $R \subseteq X \times Y$ by
    \[R := \{(x,y) \in X \times Y \mid d_Y(f(x),y) < \varepsilon\}.\]
    Since $\sup_{\ell \in L} d_Y(f(\alpha(\ell)),\beta(\ell)) < \varepsilon$ and $f(X)$ is an $\varepsilon$-net in $Y$, we have that $R$ is indeed a correspondence. Triangle inequality can be used to show that $\dis(R) < 3\varepsilon$, which implies that $d_{GH}^L(X,\alpha;Y,\beta) < \frac{3}{2}\varepsilon < 2\varepsilon$.
\end{proof}

\begin{remark}
    In the first part of the above theorem, if we assume that $\alpha$ is injective, then we can construct a $(2\epsilon,L)$-isometry from $X$ to $Y$ which has the property that $f(\alpha(\ell)) = \beta(\ell)$, for every $\ell \in L$.
\end{remark}

\begin{corollary}
\label{cor:sequence-of-epsilon-isom}
    A sequence $((X_n,\alpha_n))_{n \in \N}$ of $L$-labeled metric spaces converges to an $L$-labeled metric space $(X,\alpha)$ if and only if there exists a sequence $(\varepsilon_n)_{n \in \N}$ and a sequence of maps $f_n:X_n \to X$ such that every $f_n$ is an $(\varepsilon_n,L)$-isometry and $\varepsilon_n \to 0$.
\end{corollary}

The above corollary gives us a handy characterisation for LGH-convergence in terms of approximate isometries. In the remaining part of this section we will establish an another characterisation for LGH-convergence which uses finite approximations of labeled spaces.

\begin{definition}
    Let $(X,\alpha)$ and $(X_0,\alpha_0)$ be $L$-labeled metric spaces, where $X_0 \subseteq X$ is finite. We say that $(X_0,\alpha_0)$ is a \textbf{labeled} $\varepsilon$-\textbf{net} in $(X,\alpha)$, if $X_0$ is a $\varepsilon$-net in $X$ and furthermore $\displaystyle \sup_{\ell \in L} d_X(\alpha(\ell),\alpha_0(\ell)) < \varepsilon$.
\end{definition}

Observe that if $(X_0,\alpha_0)$ is a labeled $\varepsilon$-net in $(X,\alpha)$, then $d_{GH}^L(X,\alpha;X_0,\alpha_0) < \varepsilon$. Hence labeled nets can be seen as finite approximations of (potentially infinite) metric spaces.

\begin{definition}
\label{def:epsilon-delta-approximation}
    Let $\varepsilon, \delta > 0$ and $(X,\alpha)$ an $L$-labeled metric space. An $L$-labeled metric space $(Y,\beta)$ is an $(\varepsilon,\delta)$-\textbf{approximation} of $(X,\alpha)$ if there are finite subsets
    \[
    X_0 := \{x_i \mid 1 \leq i \leq N\} \subseteq X
    \quad\text{and}\quad
    Y_0 := \{y_i \mid 1 \leq i \leq N\} \subseteq Y,
    \]
    and functions $\alpha_0 \colon L \to X_0$ and $\beta_0 \colon L \to Y_0$ such that the following conditions hold for the finite $L$-labeled metric spaces $(X_0,\alpha_0)$ and $(Y_0,\beta_0)$.
    \begin{enumerate}
        \item \label{item:net-condition} The sets $X_0$ and $Y_0$ are labeled $\varepsilon$-nets in $X$ and $Y$ respectively.
        \item \label{item:distortion-condition} The distortion of the correspondence
        \[\{(x_i,y_i) \mid 1 \leq i \leq N\} \cup \{(\alpha_0(\ell),\beta_0(\ell)) \mid \ell \in L\}\]
        between $(X_0,\alpha_0)$ and $(Y_0,\beta_0)$ is less than $\delta$.
    \end{enumerate}
    If in addition to Items~\ref{item:net-condition} and~\ref{item:distortion-condition} the spaces $(X_0,\alpha_0)$ and $(Y_0,\beta_0)$ satisfy the condition: for all $\ell \in L$ and $1 \le i \le N$ we have $x_i = \alpha_0(\ell)$ if and only if $y_i = \beta_0(\ell)$, then $(Y,\beta)$ is a \textbf{strong} $(\varepsilon,\delta)$\textbf{-approximation} of $(X,\alpha)$.
\end{definition}

We will follow the convention that if $\varepsilon = \delta$, then $(\varepsilon,\varepsilon)$-approximations are called simply $\varepsilon$-approximations. The following theorem relates existence of approximations to LGH-distance.

\begin{theorem}\label{thm:approxmations-work}
    Let $(X,\alpha)$ a be $L$-labeled metric spaces.
    \begin{enumerate}
        \item If $(Y,\beta)$ is an $(\varepsilon,\delta)$-approximation of $(X,\alpha)$, then $d_{GH}^L(X,\alpha;Y,\beta) < 2\varepsilon + \delta/2$.
        \item Any $L$-labeled metric space $(Y,\beta)$ with $d_{GH}^L(X,\alpha;Y,\beta) < \varepsilon$ is a strong $6\varepsilon$-approximation of $(X,\alpha)$.
    \end{enumerate}
\end{theorem}
\begin{proof}
    Suppose first that $(Y,\beta)$ is an $(\varepsilon,\delta)$-approximation of $(X,\alpha)$. Let $(X_0,\alpha_0)$ and $(Y_0,\beta_0)$ be finite $L$-labeled metric spaces as in Definition~\ref{def:epsilon-delta-approximation}, where $X_0 = \{x_1,\dots,x_N\}$ and $Y_0 = \{y_1,\dots,y_N\}$. Then 
    \[R := \{(x_i,y_i) \mid 1 \leq i \leq N\} \cup \{(\alpha_0(\ell),\beta_0(\ell)) \mid \ell \in L\}\]
    is a correspondence between $(X_0,\alpha_0)$ and $(Y_0,\beta_0)$. By definition~\ref{def:epsilon-delta-approximation}, the distortion of $R$ is less than $\delta$. Thus $d_{GH}^L(X_0,\alpha_0;Y_0,\beta_0) < \frac{\delta}{2}$ by Theorem~\ref{theorem:distance-correspondence}. Since $X_0$ is a labeled $\varepsilon$-net in $X$, we have that $d_{GH}^L(X,\alpha;X_0,\alpha_0) < \varepsilon$. Analogously $d_{GH}^L(Y,\beta;Y_0,\beta_0) < \varepsilon$. Triangle inequality gives us
    \begin{equation}
    \begin{split}
        d_{GH}^L(X,\alpha;Y,\beta)
        &\le
        d_{GH}^L(X,\alpha;X_0,\alpha_0)
        +
        d_{GH}^L(X_0,\alpha_0;Y_0,\beta_0)
        +
        d_{GH}^L(Y,\beta;Y_0,\beta_0)
        \\
        &<
        2\varepsilon + \frac{\delta}{2}.
    \end{split}
    \end{equation}
    
    Suppose then that $d_{GH}^L(X,\alpha;Y,\beta) < \varepsilon$. There exists a $(2\varepsilon,L)$-isometry $f$ from $X$ to $Y$ by Theorem~\ref{thm:epsilom-isometries}. Let $X_0 = \{x_1,\dots,x_n\} \subseteq X$ be a finite $\varepsilon/2$-net. We define a labeling $\alpha_0:L \to X_0$ by setting $\alpha_0(\ell) := \alpha(\ell)$, if $\alpha(\ell) \in X_0$, and otherwise we set $\alpha_0(\ell) := x_i$ for some $x_i \in X_0$ such that $d_X(x_i,\alpha(\ell)) < \varepsilon/2$. Choosing such $x_i$ is possible since $X_0$ is an $\varepsilon/2$-net in $X$. Since $X_0$ is finite
    \[
    \sup_{\ell \in L} d_X(\alpha(\ell),\alpha_0(\ell)) \le \frac{\varepsilon}{2} < \varepsilon.
    \]
    Next, for every $1 \leq i \leq N$ let $y_i := f(x_i)$ and $Y_0 := \{y_i \mid 1 \leq i \leq N\}$. We then define a labeling $\beta_0:L \to Y_0$ by setting $y_i = \beta_0(\ell)$ iff $x_i = \alpha_0(\ell)$ for every $\ell \in L$. We claim that \[
    \sup_{\ell \in L} d_Y(\beta_0(\ell),\beta(\ell)) < 6\varepsilon.
    \]
    Let $\ell \in L$. Since $f$ is an $2\varepsilon$-isometry we have that $d_Y(\beta(\ell),f(\alpha(\ell))) < 2\varepsilon$ and
    \[
    |d_Y(f(\alpha(\ell)),f(\alpha_0(\ell))) - d_X(\alpha(\ell),\alpha_0(\ell))| < 2\varepsilon.
    \]
    We have now two cases.
    \begin{enumerate}
        \item If $d_Y(f(\alpha(\ell)),f(\alpha_0(\ell))) < d_X(\alpha(\ell),\alpha_0(\ell)) < \varepsilon$, then
        \begin{align*}
            d_Y(\beta(\ell),\beta_0(\ell)) 
            &= d_Y(\beta(\ell),f(\alpha_0(\ell)))
            \\
            &\leq d_Y(\beta(\ell),f(\alpha(\ell))) + d_Y(f(\alpha(\ell)),f(\alpha_0(\ell)))
            \\
            &< 2\varepsilon + \varepsilon = 3\varepsilon.
        \end{align*}
        \item If $d_Y(f(\alpha(\ell)),f(\alpha_0(\ell))) \geq d_X(\alpha(\ell),\alpha_0(\ell))$, then $d_Y(f(\alpha(\ell)),f(\alpha_0(\ell))) < 3\varepsilon$ and hence
        \begin{align*}
            d_Y(\beta(\ell),\beta_0(\ell)) 
            &= d_Y(\beta(\ell),f(\alpha_0(\ell)))
            \\
            &\leq d_Y(\beta(\ell),f(\alpha(\ell))) + d_Y(f(\alpha(\ell)),f(\alpha_0(\ell)))
            \\
            &< 2\varepsilon + 3\varepsilon = 5\varepsilon.
        \end{align*}
    \end{enumerate}
    Thus
    \[
    \sup_{\ell \in L} d_Y(\beta_0(\ell),\beta(\ell)) \le 5\varepsilon < 6\varepsilon.
    \]
    To conclude, we need to establish that $Y_0$ is a $6\varepsilon$-net in $Y$; this can be done as in the case of Gromov-Hausdorff distance (see~\cite[Corollary 7.3.28]{metricgeometry}).
\end{proof}

The final characterization, Theorem~\ref{thm:finite-net-condition-convergence}, for LGH-convergence says that convergence of labeled spaces can be captured by convergence on finite labeled subspaces.

\begin{theorem}
\label{thm:finite-net-condition-convergence}
    Let $(X,\alpha)$ be a $L$-labeled metric space and let $((X_n,\alpha_n))_{n \in \mathbb{N}}$ be a sequence of $L$-labeled metric spaces. Then $(X_n,\alpha_n) \to (X,\alpha)$ if and only if for every $\varepsilon > 0$ there exists
    \begin{enumerate}
        \item \label{item:net-convergence1} a finite set $S \subseteq X$ and a labeling $\beta:L \to S$ such that $(S,\beta)$ is a labeled $\varepsilon$-net in $(X,\alpha)$,
        \item \label{item:net-convergence2} and for each $n \in \N$ there exists a finite set $S_n \subseteq X_n$ and a labeling $\beta_n:L \to S_n$ such that $(S_n,\beta_n)$ is a labeled $\varepsilon$-net in $(X_n,\alpha_n)$,
    \end{enumerate}
    so that $(S_n,\beta_n) \to (S,\beta)$.
\end{theorem}
\begin{proof}
    Fix $\varepsilon > 0$ and let $\delta := \varepsilon/3 > 0$. Suppose we are given $(S_n,\beta_n)$ and $(S,\beta)$ satisfying Items~\ref{item:net-convergence1} and~\ref{item:net-convergence2} for $\delta$. We show that \[d_{GH}^L(X_n,\alpha_n;X,\alpha) < \varepsilon,\]
    for $n$ sufficiently large. By Theorem~\ref{thm:approxmations-work}, it suffices to show that $(X_n,\alpha_n)$ is a $\delta$-approximation of $(X,\alpha)$ for $n$ sufficiently large. But this is clear, as $(S_n,\beta_n) \to (S,\beta)$ implies that for $n$ sufficiently large the correspondence between $(S_n,\beta_n)$ and $(S,\beta)$ in Definition~\ref{def:epsilon-delta-approximation} Item~\ref{item:distortion-condition} has distortion less than $\delta$.
    
    Suppose then that $(X_n,\alpha_n) \to (X,\alpha)$. Now one can proceed analogously as in the proof of Theorem \ref{thm:approxmations-work}, particularly proof of Item~\ref{item:distortion-condition}. More precisely, we first fix some label $\varepsilon$-net $(S,\beta)$ in $(X,\alpha)$.  As $(X_n,\alpha_n) \to (X,\alpha)$ implies that there exists sequences $(\varepsilon_n)_{n \in \N}$ and $(f_n)_{n \in \N}$ such that $\varepsilon_n \to 0$ and each $f_n$ is a $(\varepsilon_n,L)$-isometry. For each $n$ we set $S_n := f_n(S)$. As in the proof of Theorem \ref{thm:approxmations-work}, we can now define labelings $\beta_n$ on each $S_n$ in such a way that each $(S_n,\beta_n)$ is a label $\varepsilon$-net in $(X_n,\alpha_n)$ and furthermore $(S_n,\beta_n) \to (S,\beta)$.
\end{proof}

\section{Precompactness criteria}
\label{sec:compactness-criteria}

In this section we present precompactness criteria for collections of metric spaces. The precompactness is with respect to the topology induced by the LGH-distance on $\mathcal{M}_L$ --- the space of all $L$-isometry classes --- and hence all metric spaces in this section should be understood as points in $\mathcal{M}_L$, i.e., as $L$-isometry classes.

\begin{definition}
\label{def:unif-tot-bnd}
A collection $\mathcal{X}$ of metric spaces is called \textbf{uniformly totally bounded}, if the following two conditions are satisfied.
\begin{enumerate}
    \item There is a constant $d > 0$ such that $\diam X \le d$ for all $X \in \mathcal{X}$.
    \item For any $\varepsilon > 0$ there is $N_\varepsilon \in \N$ such that any $X \in \mathcal{X}$ has an $\varepsilon$-nets consisting of at most $N_\varepsilon$ points.
\end{enumerate}
\end{definition}

\noindent Gromov's compactness theorem \cite{Gromov1981} shows that a uniformly totally bounded collection of metric spaces is precompact in the space of all isometry classes. Our main goal in this section is to extend this criteria for labeled metric spaces.

Given a sequence of metric spaces $(X_n)_{n \in \N}$ and mappings $(\alpha_n)_{n \in \N}$, where for every $n$ we have that $\alpha_n : L \to X_n$, we want to discuss the uniform convergence of the latter sequence. If $X_n \raja{GH} X$, then one reasonable way of making sense of this is to consider the disjoint union
\[\bigsqcup_{n \in \N} X_n \sqcup X\]
equipped with some metric $d$ such that $X_n \raja{H} X$ with respect to this metric. Uniform convergence of $(\alpha_n)_{n\in \N}$ is then understood to be with respect to some such metric $d$.

\begin{lemma}\label{lemma:sequence-space-compact}
    Suppose that the set
    \[Y := \bigsqcup_{n\in\N} X_n \sqcup X\]
    is equipped with a metric $d$ such that $X_n \raja{H} X$. Then $Y$ is compact.
\end{lemma}
\begin{proof}
    By assumption, we have that
    \[d_H(X_n,X) < \varepsilon_n,\]
    for some $(\varepsilon_n)_{n \in \N}$ such that $\varepsilon_n \to 0$ as $n \to \infty$. Let $(y_n)_{n\in \N}$ be a sequence of elements of $Y$. If there exists 
    \[Z \in \{X_n \mid n \in \N\} \cup \{X\}\]
    such that $y_n \in Z$ for infinitely many $n \in \N$, then $(y_n)_{n \in \N}$ has a subsequence which converges to an element in $Z$, since $Z$ is compact. Suppose then that no such $Z$ exists, in which case for every $n$ there must exists $m,m' \geq n$ such that $y_m \in X_{m'}$. Thus, for every $n$ we can choose an element $z_n$ from the sequence $(y_n)_{n \in \N}$ for which there exists an element $x_n \in X$ such that $d(z_n,x_n) < \varepsilon_n$. As $X$ is compact, $(x_n)_{n \in \N}$ must have a subsequence which converges to $x \in X$. The corresponding subsequence of $(z_n)_{n\in \N}$ must also converge to $x$, since
    \[d(z_n,x) \leq d(z_n,x_n) + d(x_n,x) < \varepsilon_n + d(x_n,x),\]
    for every $n$.
\end{proof}

Our first result is that if the set of labels is countable, then the precompactness condition of Gromov works also in the LGH-setting.

\begin{theorem}\label{thm:precompactness_countable}
Let $L$ be a countable set of labels. Then any uniformly totally bounded collection $\mathcal{X}$ of compact $L$-labeled metric spaces is precompact in the labeled Gromov-Hausdorff topology.
\end{theorem}

\begin{proof}
Pick a sequence $((X_n,\alpha_n))_{n \in \N}$ of elements of $\mathcal{X}$. We prove that it has a LGH-convergent subsequence.

Since $\{X \,:\, (X,\alpha) \in \mathcal{X}\}$ is uniformly totally bounded, we can assume --- using Gromov's compactness theorem --- that $X_n \raja{GH} X$ by possibly moving to a subsequence. In particular, there is a pseudometric $d$ on the disjoint union $Y := X \sqcup \bigsqcup_{n \in \N}X_n$ so that $X_n \raja{H} X$ with respect to $d$. Consider the sequence $(\alpha_n)_{n \in \N}$ of functions $\alpha_n \colon L \to Y$. Lemma \ref{lemma:sequence-space-compact} states that $Y$ is compact. Since $Y$ is compact and $L$ is countable, we can use the diagonal sequences trick to find a subsequence, still denoted $(\alpha_n)_{n \in \N}$, so that $(\alpha_n(\ell))_{n \in \N}$ converges for all $\ell \in L$. Define $\alpha \colon L \to Y$ by
\[
\alpha(\ell) := \lim_{n \to \infty}\alpha_n(\ell).
\]
Then $\alpha(L) \subseteq X$. Indeed, if there was $\ell \in L$ so that $\alpha(\ell) \notin X$, then for some $\delta > 0$ we have that
\[
\delta < d(\alpha(\ell),X) \le d(\alpha(\ell),\alpha_n(\ell)) + d(\alpha_n(\ell),X).
\]
By choosing $n$ large enough, we have $d(\alpha(\ell),\alpha_n(\ell)) < \delta/3$ and $d(\alpha_n(\ell),X) < \delta/3$. Such large $n$ exists, since $\alpha$ is the uniform limit of $(\alpha_n)_{n \in L}$ and $X_n \raja{GH} X$. This gives us $\delta < 2\delta/3$, a contradiction.

We conclude the proof by showing that $(X_n,\alpha_n) \raja{LGH} (X,\alpha)$. But this is now easy, since it suffices to notice that
\[
d^L_{GH}(X_n,\alpha_n;X,\alpha) \le d_H(X_n,X) + \sup_{\ell \in L}d(\alpha_n(\ell),\alpha(\ell)),
\]
where $d$ is the pseudometric on $Y$ and $d_H$ is computed with respect to $d$ as well.
\end{proof}

It is easy to come up with examples which demonstrate that in Theorem \ref{thm:precompactness_countable} one cannot drop the assumption that the set of labels is countable. For instance, let $L = [0,1]^\N$ be the set of all $[0,1]$-valued sequences and consider the compact metric space $[0,1]$ with the standard metric. For every $n$ we define a labeling $\alpha_n:L \to [0,1]$ by setting that
\[\alpha_n((x_k)_{k \in \N}) = x_n,\]
for every $(x_k)_{k \in \N} \in L$. This gives us a sequence $(([0,1],\alpha_n))_{n \in \N}$ of $L$-labeled compact metric spaces with $\{[0,1]\}$ obviously uniformly totally bounded. It is clear that $(([0,1],\alpha_n))_{n \in \N}$ does not have a subsequence which converges with respect to the LGH-metric.
        
\begin{comment}
Aiming for a contradiction, suppose that $\{([0,1],\alpha_n) \mid n \in I\}$, where $I \subseteq \N$, is a subsequence of the original sequence that converges to an $L$-labeled metric space $(X,\alpha)$. We rename the indexes to guarantee that $I = \N$. Now there must exists sequences $\{\varepsilon_n\}$ and $\{f_n\}$ such that $\varepsilon_n \to 0$ and each $f_n$ is a $(\varepsilon_n,L)$-isometry from $([0,1],\alpha_n)$ to $(X,\alpha)$. In particular, for every $\ell \in L$ we have that the sequence $\{f_n(\alpha_n(\ell))\}$ converges to $\alpha(\ell)$.
        
Pick now two labels $\ell,\ell'$ such that $\alpha(\ell) \neq \alpha(\ell')$ (clearly there must be such labels). Consider now the following sequence
\[x_n := 
\begin{cases}
\alpha_n(\ell) & \text{if } n \text{ is even} \\
\alpha_n(\ell') & \text{if } n \text{ is odd}
\end{cases}\]
Then the sequence $\{f_n(\alpha_n(\{x_n\}))\} = \{f_n(x_n)\}$ contains two subsequences which converge to distinct values and hence it itself can not converge. This is a clear contradiction.
\end{comment}

In the case where $L$ is an arbitrary set, we are faced with the problem that any precompactness condition for subsets of $\mathcal{M}_L$ must also yield a criteria for uniform convergence of functions. To see this, let $(\alpha_n)_{n \in \N}$ be a sequence of labelings $\alpha_n : L \to X$. (Note that the underlying metric space is always the same in this example.) Proposition \ref{prop:one-space} implies that $((X,\alpha_n))_{n \in \N}$ has a converging subsequence if and only if $(\alpha_n)_{n \in \N}$ has a subsequence which converges uniformily.

One could by-pass this problem by simply requiring in the precompactness condition that the labels need to converge uniformily. (After having defined formally what it means.) Instead of doing this, we prefer to assume that $L$ is a compact metric space and then formulate the precompactness criteria as an extension of the classical Arzelà–Ascoli theorem. Since we are particularly interested in the case where $L := \partial M$, $M$ being a compact manifold with boundary, the assumption that $L$ is a compact metric space is quite natural.

That our precompactness criteria works rests on the following lemma. Its proof can be obtained by extending the proof of the Arzelà–Ascoli theorem in a standard way, but for the readers convenience we present the full proof here. Similar extensions of the Arzelà–Ascoli theorem can be found, for example, in \cite[Lemma 45]{riemannianpetersen} and in \cite[Theorem 6.2]{Droniou2016}. Our variant is in particular very close to the variant mentioned in \cite[Proposition 27.20]{villani2008optimal}, where it was stated without a proof.

\begin{lemma}
    Let $(X_n)_{n \in \N}$ be a sequence of compact metric spaces such that $X_n \raja{GH} X$. Furthermore, let $(\alpha_n)_{n\in \N}$ be a sequence of (not necessarily continuous) mappings such that $\alpha_n : L \to X_n$, for every $n$. Suppose that
    \[d_{X_n}(\alpha_n(\ell),\alpha_n(\ell')) \to 0\]
    as $n \to \infty$ and $d_L(\ell,\ell') \to 0$. Then there exists a continuous mapping $\alpha:L \to X$ and a subsequence of $(\alpha_n)_{n \in \N}$ which converges uniformly to $\alpha$.
\end{lemma}
\begin{proof}
    By Lemma \ref{lemma:sequence-space-compact} the space
    \[Y := \bigsqcup_{n\in\N} X_n \sqcup X\]
    is compact, for some metric $d$ such that $X_n \raja{H} X$. As $Y$ is compact, it is complete, and hence so is $\mathcal{F}(L,Y)$ with respect to the uniform metric.
    
    Let $\{\ell_m \mid m \in \N\}$ denote a countable dense subset of $L$. As $Y$ is compact, the sets $\{\alpha_n(\ell_m) \mid n \in \N\} \subseteq Y$ are precompact for each $m$. Hence we can use a diagonalization argument to obtain a subsequence of $(\alpha_n)_{n \in \N}$, which we denote the same way, such that $(\alpha_n(\ell_m))_{n\in \N}$ converges for each $m$. As $\mathcal{F}(L,Y)$ is complete, it suffices to show that $(\alpha_n)_{n \in \N}$ is Cauchy.
    
    Fix $\varepsilon > 0$. Since we have assumed that
    \[d_{X_n}(\alpha_n(\ell),\alpha_n(\ell')) = d(\alpha_n(\ell),\alpha_n(\ell')) \to 0,\]
    as $n \to \infty$ and $d_L(\ell,\ell') \to 0$, we can find $N$ and $\delta > 0$ such that $d(\alpha_n(\ell),\alpha_n(\ell')) \leq \varepsilon$, whenever $n \geq N$ and $d_L(\ell,\ell') \leq \delta$. Let $\{\ell_1,\dots,\ell_M\}$ be a finite $\delta$-net of $\{\ell_m \mid m \in \N\}$. Now for every $n,n' \geq N$ have that
    \begin{align*}
        d(\alpha_n(\ell),\alpha_{n'}(\ell))
        &\leq
        d(\alpha_n(\ell),\alpha_n(\ell_i)) + d(\alpha_n(\ell_i),\alpha_{n'}(\ell_i)) + d(\alpha_{n'}(\ell_i),\alpha_{n'}(\ell))\\
        & \leq 2\varepsilon + d(\alpha_n(\ell_i),\alpha_{n'}(\ell_i))
    \end{align*}
    Since $\{(\alpha_n(\ell_i))_{n \in \N} \mid 1 \leq i \leq M\}$ is a finite set of converging sequences and hence we can find $N' \geq N$ such that $d(\alpha_n(\ell_i),\alpha_{n'}(\ell_i))$, for all $1\leq i \leq M$ whenever $n,n' \geq N$. This shows that $(\alpha_n)_{n \in \N}$ is Cauchy.
    
    Let $\alpha \colon L \to Y$ be the limit of $(\alpha_n)_{n \in \N}$. To conclude the proof, we show that $\alpha(L) \subseteq X$. If $\alpha(\ell) \notin X$ for some $\ell \in L$, then for some $\delta > 0$ we have that
    \[
    \delta < d(\alpha(\ell),X) \le d(\alpha(\ell),\alpha_n(\ell)) + d(\alpha_n(\ell),X).
    \]
    Then choose $n$ large enough so that $d(\alpha(\ell),\alpha_n(\ell)) < \delta/3$ and $d(\alpha_n(\ell),X) < \delta/3$. Such large $n$ exists since $\alpha$ is the uniform limit of $(\alpha_n)_{n\in\N}$ and $X_n \raja{GH} X$. This gives us that $\delta < 2\delta/3$, a contradiction.
    
    To conclude, we will show that $\alpha$ is continuous. Fix $\ell \in L$. Now, for every $\ell' \in L$ we have that
    \[d(\alpha(\ell),\alpha(\ell')) \leq d(\alpha(\ell),\alpha_n(\ell)) + d(\alpha_n(\ell),\alpha_n(\ell')) + d(\alpha_n(\ell'),\alpha(\ell')).\]
    Let $n$ be large enough so that $d(\alpha(\ell),\alpha_n(\ell)) < \varepsilon$, for every $\ell \in L$, and choose $\ell'$ in such a way that $d(\alpha_n(\ell),\alpha_n(\ell')) < \varepsilon$. Then $d(\alpha(\ell),\alpha(\ell')) < 3\varepsilon$. Thus $\alpha(\ell') \to \alpha(\ell)$ as $\ell' \to \ell$, which is what we wanted to show.
\end{proof}

\begin{theorem}
\label{thm:precompactness}
    Let $((X_n,\alpha_n))_{n \in \N}$ be a sequence of $L$-labeled compact metric spaces such that 
    \begin{enumerate}
        \item $\{X_n \mid n \in \N\}$ is uniformly totally bounded and
        \item \label{item:aa-condition} $d_{X_n}(\alpha_n(\ell),\alpha_n(\ell')) \to 0$, as $n \to \infty$ and $d_L(\ell,\ell') \to 0$.
    \end{enumerate}
    Then $((X_n,\alpha_n))_{n \in \N}$ has a converging subsequence which converges to an $L$-labeled metric space $(X,\alpha)$ where $\alpha:L \to X$ is continous.
\end{theorem}
\begin{proof}
    As $\{X_n \mid n \in \N\}$ is uniformly totally bounded, we can find a subsequence of $(X_n)_{n \in \N}$, denoted the same way, which converges to a compact metric space $X$. Applying the previous lemma we can also find a subsequence of $(\alpha_n)_{n \in \N}$, denoted again the same way, which converges uniformly to a mapping $\alpha:L \to X$. Combining these two gives us that $(X,\alpha)$ is the LGH-limit of $((X_n,\alpha_n))_{n \in \N}$.
\end{proof}

The above precompactness condition is only formulated for sequences of $L$-labeled metric spaces. The following immediate corollary gives a precompactness condition for collections of $L$-labeled metric spaces.

\begin{corollary}\label{corollary:precompactness_collections}
    Let $\mathcal{X}$ be a uniformly totally bounded set of compact $L$-labeled metric spaces. Suppose that for every $\ell \in L$ and $\varepsilon > 0$ there exists a neighborhood $U_\ell \subseteq L$ of $\ell$ such that for every $\ell' \in U_\ell$ and $(X,\alpha) \in \mathcal{X}$ we have that $d_X(\alpha(\ell),\alpha(\ell')) < \varepsilon$. Then $\mathcal{X}$ is precompact in $\mathcal M_L$.
\end{corollary}

\begin{remark}
    The requirement imposed on labelings in the statement of Corollary \ref{corollary:precompactness_collections} is essentially an equicontinous-requirement.
\end{remark}

The potential downside of Corollary \ref{corollary:precompactness_collections} is that it only works for $L$-labeled metric spaces where the labelings are continuous, while Theorem \ref{thm:precompactness} does not require this. It seems to be non-trivial to come up with a variant of Theorem \ref{thm:precompactness} which makes also sense for collections of $L$-labeled metric spaces and which does not assume that the labelings are continuous.

\begin{comment}
\begin{example}
Let $\X$ be a uniformly totally bounded set of compact metric spaces. Equip each $X \in \X$ with a labeling $\alpha_X \colon L \to X$. According to Theorem~\ref{thm:precompactness} the set
\[
\X_\alpha := \{(X,\alpha_X)\,:\, X \in \X\}
\]
is precompact in $\mathcal M_L$ provided that there exists a sequence $(\delta_n)_{n \in \N}$ of positive real numbers with the property that $\delta_n \to 0$ and that for every $n$ there exists a finite set $\Y \subseteq \X_\alpha$ such that if $(X,\alpha_X) \in (\X_\alpha \backslash \Y)$ and $d_L(\ell,\ell') < \delta_n$, then $d_X(\alpha_X(\ell),\alpha_X(\ell')) < 1/n$.
\end{example}
\end{comment}

\begin{comment}
\textcolor{red}{
\begin{corollary}
    Let $\mathcal{X} = \{(X_i,\alpha_i)\}_{i \in I}$ be a collection of $L$-labeled compact metric spaces such that
    \begin{enumerate}
        \item $\{X_i \mid i \in I\}$ is uniformly totally bounded and
        \item every sequence of elements of $\mathcal{X}$ has a subsequence $((Z_n,\alpha_n))_{n \in \N}$ such that 
        \[d_{Z_n}(\alpha_n(\ell),\alpha_n(\ell')) \to 0,\]
        as $n \to \infty$ and $d_L(\ell,\ell') \to 0$.
    \end{enumerate}
    Then $\mathcal{X}$ is precompact.
\end{corollary}
}
\end{comment}

\section{Applications to inverse problems}\label{sec:inverse-problems}

In this section we consider applications of the precompactness criteria in Corollary~\ref{corollary:precompactness_collections} to inverse problems on labeled metric spaces. The general objective in geometric inverse problems is to reconstruct a space~$X \in \mathcal X$ from some boundary data~$D \in \mathcal D$. With metric space variants the boundary data naturally consists of some restricted set of distances. On top of reconstructibility we are generally interested in stability of the reconstruction. If two spaces far away from each other can produce almost identical data, then a problem is somewhat instable, i.e., the reconstruction or inversion is not continuous? We begin by formulating a general stabilization procedure for geometric problems (cf.~\cite[Section~4.3.]{stabilitything}). Stabilization procedure can be applied to prove unquantified stability statements for a wide class of geometric inverse problems.

Consider a continuous forward map $\mathcal F \colon \mathcal X \to \mathcal D$ assumed to be injective, i.e., the reconstructibility question is already settled. Usually, the parameter space $\mathcal X$ is a topological space and the data space $\mathcal D$ a Hausdorff space. Suppose that we have gained some additional a prior information; we want to reconstruct an element in a precompact subset $\mathcal X_0 \subseteq \mathcal X$. We know that the reconstruction is going to be stable in the sense that if the data associated to objects $X,Y \in \mathcal X_0$ are close then $X$ and $Y$ are close to begin with. Indeed, the map $\mathcal F \colon \overline{\mathcal X_0} \to \mathcal D$ is continuous and injective from a compact space to a Hausdorff space having, hence, a continuous inverse when codomain is restricted to the image $\mathcal F(\overline{\mathcal X_0})$.

As a concrete example let us consider a metric space variant of the so called \textbf{travel time inverse problem}. For travel time problems in Riemannian and Finslerian geometries see~\cite{finslertraveltime,multisourcelgh,stabilitytravelsimple}.

In travel time problems the objective is to reconstruct a manifold with a boundary, or more generally an $L$-labeled metric space, from the knowledge of boundary distances. In this section we denote $\partial X := \alpha(X)$ for an $L$-labeled metric space to emphasis that $\alpha(X)$ acts as the geometric boundary of the metric space $X$. Now we formulate the metric space variant in more concrete terms. Our notations and treatment of the travel time problem follow closely that of~\cite{stabilitytravelsimple}, but in a metric rather than Riemannian set up.

\begin{definition}
\label{def:travel-time-problem}
For any $x \in X$ in an $L$-labeled compact metric space $(X,\alpha)$ we define the map
\[
r_x \colon \partial X \to \R, \quad r_x(z) := d_X(x,z).
\]
and then to $(X,\alpha)$ we associate the \textbf{travel time map} defined by
\[
R \colon X \to C(\partial X), \quad R(x) := r_x.
\]
The \textbf{travel time data} of $X$ is the set
\[
R(X)
=
\{\,
r_x
\,:\,
x \in X
\,\}.
\]
\end{definition}

The \textbf{(metric) travel time problem} asks whether an $L$-labeled metric space $X$ can be reconstructed from the knowledge of its travel time data $R(X)$. In other words, is the map $\mathcal R \colon \mathcal X \to \mathcal D$ defined by $\mathcal R(X) = \overline{R(X)}$ injective, where $\overline{R(X)}$ is the closure of $R(X) \subseteq C(\partial X)$ with respect to the supremum norm $\norm{\cdot}_{\infty}$. Here $\mathcal X \subseteq \mathcal M_L$ is some set of compact $L$-labeled metric spaces and $\mathcal D$ is the set of compact metric spaces equipped with the Gromov-Hausdorff distance.

\begin{remark}
For any compact $L$-labeled metric space $X$ the travel time data $\overline{R(X)}$ is a compact space and thus in $\mathcal D$. This follows from Arzela-Ascoli theorem, since $\norm{r_x}_\infty$ is at most the diameter of $X$ for all $r_x \in R(X)$, and the set $R(X)$ is equicontinuous as all $r_x$ are $1$-Lipschitz functions.
\end{remark}

We give a positive solution to the travel time problem in a certain class $\mathcal X \subseteq \mathcal M_L$ of $L$-labeled length metric spaces and prove that the inversion is stable in the sense that the travel time data map $\mathcal R$ is continuously invertible.

We begin with the definition of a simple length metric space. In a length metric space $l(\gamma)$ denotes the length of a path $\gamma$ and geodesics are defined to be the path that are locally distance minimizing.

\begin{definition}
\label{def:simplicity}
Let $X$ be a compact length metric space. Equip $X$ with a labeling $\alpha \colon L \to X$. The pair $(X,\alpha)$ is called \textbf{simple $L$-labeled metric space}, if
\begin{enumerate}
    \item \label{item:non-trapping} any geodesic segment in $X$ can be extended to a finite length geodesic segment with end points in $\partial X$,
    \item \label{item:unique-geodesics} for any two points $x$ and $y$ in $X$ there exists a unique geodesic $\gamma_{xy}$ of $X$ whose end points are $x$ and $y$,
    \item \label{item:convex-bnd} for any $x \in X$ the function $\tau_x \colon \partial X \to \R$ defined by $\tau_x(z) = l(\gamma_{xz})$ is continuous. Here $\gamma_{xz}$ is the unique geodesic between $x$ and $z$.
\end{enumerate}
\end{definition}

Note that since in a simple $L$-labeled metric space geodesics are globally unique, we have $d(x,y) = l(\gamma_{xy})$ for all $x,y \in X$. Thus $\tau_x = r_x$, where $r_x$ is as in definition~\ref{def:travel-time-problem}. We will use notation $r_x$ for these functions.

\begin{remark}
Any simple (see~\cite[Chapter~3.8.]{paternain2023geometric}) Riemannian manifold $(M,g)$ is a simple $\partial M$-labeled metric space if $M$ is equipped with the labeling $i \colon \partial M \to M$, $i(x) = x$. To this regard, simple $L$-labeled metric spaces are metric space analogs of simple Riemannian manifolds.
\end{remark}

\red{}

\begin{lemma}
\label{lma:toy-injectivity}
The travel time map $R \colon X \to C(\partial X)$ of any simple $L$-labeled metric $(X,\alpha)$ is an isometric embedding of $X$ into $(C(\partial X),\norm{\cdot}_\infty)$. Particularly, the closure $\overline{R(X)} \in \mathcal D$ determines the $L$-labeled space $X$ uniquely up to isometry.
\end{lemma}

\begin{proof}
First, note that by uniqueness of geodesics $l(\gamma_{xy}) = d(x,y)$ for any two points $x,y \in X$. Thus by triangle inequality it holds that
\[
\abs{r_x(z) - r_y(z)} \le d(x,y)
\]
for all $x,y \in X$ and $z \in \partial X$. Thus
\[
\norm{R(x) - R(y)}_\infty
=
\sup_{z \in \partial X}\abs{r_x(z) - r_y(z)}
\le
d(x,y).
\]

To get the opposite estimate let $x,y \in X$ and extend the unique geodesic $\gamma_{xy}$ to finite length geodesic $\tilde\gamma_{xy}$ between two boundary points $z,z' \in \partial X$ using condition~\ref{item:non-trapping} of simplicity. Since the geodesic $\tilde\gamma_{xy}$ is length minizing by condition~\ref{item:unique-geodesics} of simplicity, we see that the segment of $\tilde\gamma_{xy}$ between $x$ and $z$ is $\gamma_{xz}$ and the segment of $\tilde\gamma_{xy}$ between $x$ and $z'$ is $\gamma_{xz'}$. Thus either $x$, $y$ and $z$ or $x$, $y$ and $z'$ lay on a same unique minimizing geodesic segment. Say $y$ is on the segment $\gamma_{xz}$. Then
\[
\abs{r_x(z) - r_y(z)}
=
\abs{d(x,z)-d(y,z)}
=
d(x,y)
\]
and evidently
\[
\norm{R(x) - R(y)}_\infty \ge d(x,y).
\]
\end{proof}

Let $\mathcal S \subseteq \mathcal M_L$ be the subspace of $L$-isometry classes of all simple $L$-labeled metric spaces. Lemma~\ref{lma:toy-injectivity} says that we can reconstruct a space $(X,\alpha) \in \mathcal S$ up to isometry from $\overline{R(X)}$. In other words, the travel time data map $\mathcal R \colon \mathcal S \to \mathcal D, \mathcal R(X) = R(X)$ is injective. The general stabilization scheme described in the beginning of thus section can now be used to stabilize the problem, i.e., to prove that if data coming from two spaces are close then the spaces have to be close.

\begin{theorem}
\label{thm:stability}
Let $\mathcal X \subseteq \mathcal M_L$ be a set of simple $L$-labeled metric space satisfying conditions of Corollary~\ref{corollary:precompactness_collections}. Suppose that the LGH-closure $\overline{\mathcal X}$ of $\mathcal X$ consists of simple $L$-labeled metric spaces. Then the travel time data map $\mathcal R \colon \overline{\mathcal X} \to \mathcal D$ is a homeomorphism onto its image. Particularly, the inverse $\mathcal R^{-1} \colon \mathcal R(\overline{\mathcal X}) \to \overline{\mathcal X}$ is continuous.
\end{theorem}

\begin{proof}
The maps $\mathcal R$ is continuous, since for any $X \in \mathcal S$ the space $X$ and $R(X)$ are isometric by lemma~\ref{lma:toy-injectivity} and hence
\[
d_{GH}(\mathcal R(X),\mathcal R(Y))
=
d_{GH}(R(X),R(Y))
=
d_{GH}(X,Y)
\le
d^L_{GH}(X,Y)
\]
for all $X,Y \in \mathcal S$. Furthermore, the set $\mathcal{X}$ is precompact by corollary~\ref{corollary:precompactness_collections}. Thus $\mathcal R \colon \overline{\mathcal X} \to \mathcal D$ induces a homeomorphism onto its image, since it is a continuous and injective map from a compact space to a Hausdorff space.
\end{proof}

\begin{remark}
In Theorem~\ref{thm:stability}, the assumption that the LGH-closure $\overline{\mathcal X}$ consists of only simple $L$-labeled metric spaces appears quite strong. However, a prime example for a simple space is a simple Riemannian manifold, and this simplicity can be quantified by certain geometric constants. We suspect that if a collection $\mathcal X$ satisfying assumptions of Corollary~\ref{corollary:precompactness_collections} consists of simple Riemannian manifolds with geometry bounded as in ~\cite[Definition~6]{multisourcelgh} by the same constants uniformly, then the LGH-closure consists of only simple Riemannian manifolds. In general, the closure will not contain only simple spaces, or simple Riemannian manifolds, since simplicity is not in general preserved under geometric limits.
\end{remark}

\bibliographystyle{plainurl}
\bibliography{arxiv}

\appendix
\section{Equivalence of the two definitions of LGH-distance}\label{appendix:equivalence-of-definitions}

Our definition of the labeled Gromov-Hausdorff distance deviates slightly from that of \cite{multisourcelgh}, where the authors consider the metric
\begin{equation}
\label{eq:old-metric}
d_L(X,\alpha;Y,\beta)
=
\inf_{Z,f,g}
\left\{
d^Z_H(f(X),g(Y)) + \sup_{\ell \in L}d(f(\alpha(\ell)),g(\beta(\ell)))
\right\}.
\end{equation}
Here the infimum is taken over all (compact) metric spaces $Z$ and all isometric embeddings $f \colon X \to Z$ and $g \colon Y \to Z$.
We prove that the definitions are bi-Lipschitz equivalent.

\begin{proposition}
Let $L$ be a set and let $(X,\alpha)$ and $(Y,\beta)$ be $L$-labeled metric spaces. Then
\[
d^L_{GH}(X,\alpha;Y,\beta)
\le
d_L(X,\alpha;Y,\beta)
\le
2d^L_{GH}(X,\alpha;Y,\beta),
\]
where $d_L$ is the metric in~\eqref{eq:old-metric}.
\end{proposition}

\begin{proof}
First, let $r := d^L_{GH}(X,\alpha;Y,\beta)$. Then for any $\varepsilon>0$ there is a pseudometric $d$ on $Z := X \sqcup Y$ so that $d_H(X,Y) \le r + \varepsilon/2$ and $\sup_{\ell \in L}d(\alpha(\ell),\beta(\ell)) \le r + \varepsilon/2$. Let $f \colon X \to Z$ and $g \colon Y \to Z$ be the inclusions. Then
\[
d_L(X,\alpha;Y,\beta)
\le
d^Z_H(f(X),g(Y)) + \sup_{\ell \in L}d(f(\alpha(\ell)),g(\beta(\ell)))
\le
2r + \varepsilon.
\]
Since this holds for any $\varepsilon > 0$, we see that
\[
d_L(X,\alpha;Y,\beta)
\le
2r
=
2d^L_{GH}(X,\alpha;Y,\beta).
\]

To prove the other inequality we let $r := d_L(X,\alpha;Y,\beta)$. Now for any $\varepsilon > 0$ there is a metric space $Z$ and isometric embeddings $f \colon X \to Z$ and $g \colon Y \to Z$ so that
\[
d^Z_H(f(X),g(Y)) + \sup_{\ell \in L}d_Z(f(\alpha(\ell)),g(\beta(\ell))) < r + \varepsilon.
\]
We define a pseudometric $d$ on $Z := X \sqcup Y$ by $d|_{X \times X} = d_X$, $d|_{Y \times Y} = d_Y$ and
\[
d(y,x)
=
d(x,y)
=
d_Z(f(x),g(y))
\quad\text{for all }
x \in X
\text{ and }
y \in Y.
\]
Then since
\[
d_H(X,Y) = d^Z_H(f(X),g(Y)) \le r + \varepsilon
\]
and
\[
\sup_{\ell \in L}
d(\alpha(\ell),\beta(\ell))
=
\sup_{\ell \in L}
d_Z(f(\alpha(\ell)),g(\beta(\ell)))
\le
r + \varepsilon,
\]
we have
\[
d^L_{GH}(X,\alpha;Y,\beta) \le r + \varepsilon.
\]
Again, since $\varepsilon>0$ was arbitrary, we have shown that
\[
d^L_{GH}(X,\alpha;Y,\beta)
\le
r
=
d_L(X,\alpha;Y,\beta)
\]
concluding the proof.
\end{proof}

\section{Number of L-isometry classes}\label{appendix:no-isometry-classes}

\begin{lemma}\label{lemma:dense_subsets_determine_isometryclass}
    For every $L$-labeled metric space $(X,\alpha)$ there exists a dense set $\alpha(L) \subseteq N \subseteq X$ such that $|N| \leq \max\{\aleph_0,|L|\}$ and for every $L$-labeled metric space $(Y,\beta)$ we have that $(X,\alpha)$ is $L$-isometric to $(Y,\beta)$ if and only if there exists a dense set $\beta(L) \subseteq M \subseteq Y$ such that $(N,\alpha)$ is $L$-isometric to $(M,\beta)$.
\end{lemma}
\begin{proof}
    Since $(X,\alpha)$ is a compact metric space, there exists a countable dense set $N_0 \subseteq X$. We claim that $N := N_0 \cup \alpha(L)$ is the desired set. Clearly $|N| \leq \max\{\aleph_0,|L|\}$ as $N_0$ is countable. Furthermore, it is clear that if $(Y,\beta)$ is $L$-isometric to $(X,\alpha)$ via $h$, then $(N,\alpha)$ is $L$-isometric to $(h(N),\beta)$.
    
    Consider then an $L$-labeled metric space $(Y,\beta)$ which contains a dense set $\beta(L) \subseteq M \subseteq Y$ such that $(N,\alpha)$ is $L$-isometric to $(M,\beta)$ via $h$. We claim that $(X,\alpha)$ is $L$-isometric to $(Y,\beta)$. Consider the mapping $g:X \to Y$ which is defined by
    \[g(x) = \lim_{n \to \infty} h(x_n),\]
    where $(x_n)_{n\in \N}$ is an arbitrary sequence of elements of $N$ such that $x_n \to x$. To see that this mapping is well-defined, we observe the following facts.
    \begin{enumerate}
        \item Since $Y$ is complete, we have that 
        \[\lim_{n \to \infty} h(x_n) \in Y\] for every sequence $(x_n)_{n \in \N}$ which converges in $X$.
        \item If $x_n \to x$ and $y_m \to x$, then using the continuity of $d_Y$ we have that
        \begin{align*}
            d_Y(\lim_{n \to \infty} h(x_n), \lim_{m \to \infty} h(y_m))
            &=
            \lim_{n \to \infty}\lim_{m \to \infty} d_Y(h(x_n),h(y_m))
            \\
            &=
            \lim_{n \to \infty}\lim_{m \to \infty} d_X(x_n,y_m) =  0,
        \end{align*}
        since $h$ is an isometry. Hence the values of $g$ do not depend on what converging sequences we choose.
    \end{enumerate}
    We claim that $g$ is an $L$-isometry.
    \begin{enumerate}
        \item To see that $g$ is a surjection, let $y \in Y$. Pick a sequence $(y_n)_{n \in \N}$ of elements of $M$ such that $y_n \to y$. Since $h$ is a bijection, for every $n \in \N$ we can find an element $x_n \in N$ such that $h(x_n) = y_n$. Since $h$ is an isometry, the sequence $(x_n)_{n \in \N}$ is Cauchy, and since $X$ is compact, there exists $x \in X$ such that $x_n \to x$. Then $g(x) = y$.
        \item To see that $g$ is an isometry, let $x,y \in X$. Using the continuity of $d_X$ and $d_Y$ and the fact that $h$ is an isometry we have that
        \begin{align*}
            d_Y(g(x),g(y)) &= d_Y(\lim_{n \to \infty} h(x_n), \lim_{m \to \infty} h(y_m)) \\
            &= \lim_{n \to \infty} \lim_{m \to \infty} d_Y(h(x_n),h(y_m)) \\
            &= \lim_{n \to \infty} \lim_{m \to \infty} d_X(x_n,y_m) \\
            &= d_X(\lim_{n \to \infty} x_n, \lim_{m \to \infty} y_m) \\
            &= d_X(x,y)
        \end{align*}
        \item Since $\alpha(L) \subseteq N$ and the restriction of $g$ to $N$ is $h$ we have that
        \[g(\alpha(\ell)) = h(\alpha(\ell)) = \beta(\ell),\]
        for every $\ell \in L$, as $h$ is an $L$-isometry.
    \end{enumerate}
    This concludes our proof that $g$ is an $L$-isometry.
\end{proof}

\begin{theorem}
    The number of distinct $L$-isometry classes is bounded from above by $2^{\max\{\aleph_0,|L|\}}$. In particular, the class $\mathcal{M}_L$ is a set.
\end{theorem}
\begin{proof}
    Lemma \ref{lemma:dense_subsets_determine_isometryclass} implies that one can associate to each $\mathcal{X} \in \mathcal{M}_L$ a metric space $(N,\alpha)$ of size at most $\max \{\aleph_0,|L|\}$ such that an $L$-labeled metric space $(Y,\beta)$ belongs to $\mathcal{X}$ if and only if it contains a subspace which is isometric to $(N,\alpha)$. In other words there exists an injection from $\mathcal{M}_L$ to the set $\mathcal{M}_L(\max\{\aleph_0,|L|\})$ of all isometry classes of $L$-labeled metric spaces of cardinality at most $\max \{\aleph_0,|L|\}$. Thus an upper bound on the cardinality of $\mathcal{M}_L(\max\{\aleph_0,|L|\})$ will give us an upper bound on the cardinality of $\mathcal{M}_L$.
    
    To bound the cardinality of $\mathcal{M}_L(\max\{\aleph_0,|L|\})$, we will instead bound the number of $L$-labeled metric spaces of cardinality at most $\max\{\aleph_0,|L|\}$ with the assumption that for every cardinal $\kappa \leq \max\{\aleph_0,|L|\}$ all metric spaces of size $\kappa$ have the same underlying set. We will use standard results on the arithmetic of cardinal numbers, see e.g. \cite[Section 24]{halmos2017naive}
    
    We start by observing that for every cardinal $\kappa \leq \max\{\aleph_0,|L|\}$ there are (over a fixed set of size $\kappa$) at most
    \[|\R_{\geq 0}|^{|\kappa \times \kappa|} \cdot \kappa^{|L|} = 2^{\aleph_0 \cdot |\kappa \times \kappa|} \cdot \kappa^{|L|}\]
    many $L$-labeled metric spaces of cardinality $\kappa$. This expression can be bounded from above by
    \[2^{\max\{\aleph_0,|L|\}} \cdot \max\{\aleph_0,|L|\}^{|L|}.\]
    A straightforward calculation shows that $\max\{\aleph_0,|L|\}^{|L|}$ is either $2^{\max\{\aleph_0,|L|\}}$ or $\aleph_0$. Regardless of what it is, the above expression simplifies to $2^{\max\{\aleph_0,|L|\}}$. Hence the total number of $L$-labeled metric spaces of cardinality at most $\max \{\aleph_0,|L|\}$ can be bounded from above by
    \[2^{\max\{\aleph_0,|L|\}} \cdot \max\{\aleph_0,|L|\} = 2^{\max\{\aleph_0,|L|\}}\]
    which is what we wanted to show.
\end{proof}

\section{Compactness property}\label{appendix:compactness-property}

Consider an $L$-labeled metric space $(X,\alpha)$. So far we have not made any explicit requirements on the topology of the set $\alpha(L) \subseteq X$. Keeping in mind the intuition that $L$ can be viewed as the boundary of a compact manifold, it is reasonable to assume, for example, that $\alpha(L)$ should be closed. Here we give an example of what this particular assumption entails.

\begin{lemma}\label{lemma:distances-preserved}
Let $(X,\alpha)$ and $(Y,\beta)$ be $L$-labeled metric spaces. If \[d^{\{\ell,\ell'\}}_{GH}(X,\alpha;Y,\beta) = 0\]
for every $\ell,\ell' \in L$, then
\[d_X(\alpha(\ell),\alpha(\ell')) = d_Y(\beta(\ell),\beta(\ell')),\]
for every $\ell, \ell' \in L$.
\end{lemma}
\begin{proof}
    If $L' := \{\ell,\ell'\} \subseteq L$, then by assumption $d^{L'}_{GH}(X,\alpha;Y,\beta) = 0$. Since $X$ and $Y$ are compact, Proposition \ref{prop:l-isometric} implies that there is an $L'$-isometry $h \colon X \to Y$. Since $h$ is an $L'$-isometry, we have that
    \[
    d_X(\alpha(\ell),\alpha(\ell'))
    =
    d_Y(h(\alpha(\ell)),h(\alpha(\ell')))
    =
    d_Y(\beta(\ell),\beta(\ell')),
    \]
    which is what we wanted to show.
\end{proof}

\begin{theorem}
Let $(X,\alpha)$ and $(Y,\beta)$ be $L$-labeled metric spaces. Suppose that for all finite subsets $L' \subseteq L$ we have $d^{L'}_{GH}(X,\alpha;Y,\beta) = 0$. Suppose furthermore that $\alpha(L) \subseteq X$ is closed. Then $d^L_{GH}(X,\alpha;Y,\beta) = 0$.
\end{theorem}
\begin{proof}
We will prove that $d_{GH}^L(X,\alpha;Y,\beta) \leq \varepsilon$, for every $\varepsilon > 0$, from which the claim follows. Fix $\varepsilon > 0$ and consider the following open covering of $\alpha(L)$:
\[\bigcup_{\ell \in L} B\bigg(\alpha(\ell),\frac{\varepsilon}{3}\bigg).\]
Since $\alpha(L)$ is closed and $X$ is compact, there exists a finite set $L' := \{\ell_1,\dots,\ell_N\}$ such that
\[\alpha(L) \subseteq B\bigg(\alpha(\ell_1),\frac{\varepsilon}{3}\bigg) \cup \dots \cup B\bigg(\alpha(\ell_N),\frac{\varepsilon}{3}\bigg).\]
By assumption, $d_{GH}^{L'}(X,\alpha;Y,\beta) = 0$. Thus there exists a $(\frac{\varepsilon}{3},L')$-admissible metric $d$ on $X \sqcup Y$. We claim that it is in fact $(\varepsilon,L)$-admissible. Note first that $d_H(X,Y) \leq \frac{\varepsilon}{3} < \varepsilon$. Let then $\ell \in L$. Now there exists $\ell_i \in L'$ such that $d_X(\alpha(\ell),\alpha(\ell_i)) < \frac{\varepsilon}{3}$. From Lemma \ref{lemma:distances-preserved} it follows that $d_Y(\beta(\ell),\beta(\ell_i)) < \frac{\varepsilon}{3}$. Thus
\[d(\alpha(\ell),\beta(\ell)) \leq d(\alpha(\ell),\alpha(\ell_i)) + d(\alpha(\ell_i),\beta(\ell_i)) + d(\beta(\ell_i),\beta(\ell)) < \frac{\varepsilon}{3} +  \frac{\varepsilon}{3} + \frac{\varepsilon}{3} = \varepsilon.\]
Hence $d_{GH}^L(X,\alpha;Y,\beta) \leq \varepsilon$.
\end{proof}

\begin{remark}
    If $L$ is a compact space, then instead of assuming that $\alpha(L) \subseteq X$ is closed one can of course just assume that $\alpha$ is continuous.
\end{remark}

\end{document}